\documentclass{article}
\usepackage{arxiv}
\usepackage[utf8]{inputenc} 
\usepackage[T1]{fontenc}    
\usepackage{hyperref}       
\usepackage{url}            
\usepackage{booktabs}       
\usepackage{amsfonts}       
\usepackage{nicefrac}       
\usepackage{microtype}      

\usepackage{amssymb,amsmath,amsthm}
\usepackage{enumitem}
\usepackage{hyperref}

\def\dom{\mathop{\rm dom}\nolimits}
\def\int{\mathop{\rm int}\nolimits}
\def\cl{\mathop{\rm cl}\nolimits}
\def\img{\mathop{\rm img}\nolimits}
\def\cf{\mathop{\rm cf}\nolimits}

\makeatletter
\def\moverlay{\mathpalette\mov@rlay}
\def\mov@rlay#1#2{\leavevmode\vtop{%
   \baselineskip\z@skip \lineskiplimit-\maxdimen
   \ialign{\hfil$\m@th#1##$\hfil\cr#2\crcr}}}
\newcommand{\charfusion}[3][\mathord]{
    #1{\ifx#1\mathop\vphantom{#2}\fi
        \mathpalette\mov@rlay{#2\cr#3}
      }
    \ifx#1\mathop\expandafter\displaylimits\fi}
\makeatother

\newcommand{\cupdot}{\charfusion[\mathbin]{\cup}{\cdot}}
\newcommand{\bigcupdot}{\charfusion[\mathop]{\bigcup}{\cdot}}

\renewcommand{\restriction}{\mathord{\upharpoonright}}

\newtheorem{theorem}{Theorem}[section]
\newtheorem{lemma}[theorem]{Lemma}
\newtheorem{proposition}[theorem]{Proposition}
\newtheorem{corollary}[theorem]{Corollary}

\theoremstyle{definition}
\newtheorem{definition}[theorem]{Definition}

\numberwithin{equation}{section}

\theoremstyle{remark}
\newtheorem{remark}[theorem]{Remark}
\newtheorem*{nclaim}{Claim}
\newtheorem{claim}{Claim}[theorem]
\newenvironment{claimproof}[1]{\par\noindent\textit{Proof of claim:}\space#1}{\hfill $\blacksquare$\bigskip}

\def\dom{\mathop{\rm dom}\nolimits}
\def\img{\mathop{\rm img}\nolimits}
\def\ON{\mathop{\rm ON}\nolimits}

\title{The Open Coloring Axiom}
\author{Tonatiuh Matos-Wiederhold}
\date{January 27, 2022}

\begin{document}

\maketitle

\begin{abstract}
 This work is concerned with an axiom introduced by Todorc\v{e}vi\'{c} in \cite{stevo} that constitutes a Ramsey-like statement regarding the topology of the reals. Our aim is to explain the axiom in detail, give some interesting applications and finally prove that the axiom is indeed consistent with ZFC, so that it makes sense to consider working with it in the first place.
 
 For this particular academic endeavour, we cover several advanced topics in set theory, including concepts like {\sl Hausdorff gaps}, forcing, infinitary combinatorics and a tad of topology. We employ, for example, an argument based on Rothberger's theorem to show that the Open Coloring Axiom implies the equality $\mathfrak b=\aleph_2$, which in turn makes this axiom inconsistent with CH. In other words, in ZFC, the Open Coloring Axiom could be false. To prove its relative consistency, we show that the axiom could be true by following a rather long and technical lemma of Todorc\v{e}vi\'{c}, which leads to the culmination of this work.
\end{abstract}

\section{Preliminaries}

 All the notation and basic results regarding set theory not explicitly defined here must be understood as in \cite{hernandez}.
 
 The symbol $\omega$ denotes the first transfinite ordinal number and $\mathfrak c$ the cardinality of the set of real numbers, i.e., $\mathfrak c=|\mathbb R|=2^{\aleph_0}$. A set with exactly $\mathfrak c$ elements is also said to be of size {\sl continuum}. For us, a set is {\sl countable} if it has at most $\aleph_0$ elements. If $\kappa$ is a cardinal number, $[A]^\kappa$ is the family of all subsets of $A$ having exactly $\kappa$ elements. Formally, $$[A]^\kappa:=\{B\subseteq A\colon |B|=\kappa\}.$$ In particular, $[A]^2$ is the set of all unordered pairs of elements of $A$. The class of all ordinal numbers is denoted by $\ON$.
 
 Given two sets $A$ and $B$, the collection of functions from $A$ to $B$ will be denoted by the symbol $B^A$; if both $A$ and $B$ happen to be cardinal numbers, we may choose the notation ${}^AB$ instead to avoid confusion (since, for example, $\mathfrak c^n$ and $\mathfrak c$ are equal in cardinal arithmetic but are different as sets of functions). In particular, if $\alpha$ is an ordinal number, then $A^{<\alpha}:=\bigcup_{\beta<\alpha}A^\beta$. For example, $\omega^{<\omega}$ is the set of all finite sequences of natural numbers. We will use the set-theoretical definition of {\sl function}. Hence, for a function $f$, $\dom(f)$ and $\img(f)$ denote the domain and image of $f$, respectively; also, $f[A]:=\{f(a)\colon a\in A\cap\dom(f)\}$ and $f^{-1}[A]:=\{a\in\dom(f)\colon f(a)\in A\}$. In particular, if $y$ is an element of the codomain of the function $f$, then the {\sl fiber of $y$ under $f$} is the set $f^{-1}[\{y\}]$. Finally, the {\sl restriction} of $f$ to $A$ is defined by $f\restriction A:=\{(a,b)\in f\colon a\in A\}$.
 
 We use the symbol $\bigcupdot$ to denote {\sl disjoint union}, that is, $A=\bigcupdot B$ means that $A=\bigcup B$ and that, for any two distinct elements of $B$, their intersection is empty. Similarly, the disjoint union of just two sets is written as $A\cupdot B$.
 
 A {\sl topological space} will be a pair $(X,\tau)$ where $X$ is a non-empty set and $\tau$ (the {\sl topology}) is a family of subsets of $X$ that includes both $X$ and $\emptyset$ and is closed under (arbitrary) unions and finite intersections. Elements of $\tau$ are called {\sl open} subsets of $(X,\tau)$. Whenever there's no risk of ambiguity, we may refer to the topological space simply by $X$ and we use both $\tau(X)$ and $\tau_X$ to denote its respective topology. As usual, for any point $x\in X$, $\tau_X(x)$ the collection of open sets that contain $x$.
 
 All basic results and notation in topology can be found in \cite{tamariz}, including for instance, the definition and basic properties of the {\sl product topology} and the topological {\sl closure of a set $A$ in the space $X$}, which we denote by $\cl_X(A)$. We will use the fact that $\omega^\omega$, seen as a product of countably many discrete spaces, has the following properties (see \cite[Proposición~1.15, p.~9]{dania} for details): given $t\in\omega^{<\omega}$, $[t]:=\{f\in\omega^\omega\colon t\subseteq f\}$ is an open subset of $\omega^\omega$ and $\{[s]\colon s\in\omega^{<\omega}\}$ is a countable basis for $\omega^\omega$. 

 The following combinatorial lemmas will prove useful later on as part of long proofs. We prove them here as an attempt to keep said proofs free of distracting arguments.
 
\begin{lemma}\label{lemma:injection}
 If $\psi\colon\omega_1\to\omega_1$ is an injection, then there is an uncountable set $E\subseteq\omega_1$ such that $\psi\restriction E$ is (strictly) increasing.
\end{lemma}

\begin{proof}
 We construct $E$ by recursion on $\omega_1$. Suppose that for $\vartheta<\omega_1$ we have an increasing sequence $\{e_\eta\colon\eta<\vartheta\}\subseteq\omega_1$ on which $\psi$ is strictly increasing (i.e., $\xi<\eta<\vartheta$ implies $\psi(e_\xi)<\psi(e_\eta)$). Since $e:=\sup_{\eta<\vartheta}e_\eta+1$ is a countable ordinal, $\omega_1\setminus e$ is an uncountable subset of $\omega_1$ and (by injectivity) so is $\psi[\omega_1\setminus e]$. Thus, this set is cofinal in $\omega_1$, which allows us to pick out an element $y\in\psi[\omega_1\setminus e]$ such that $\psi[e]\subseteq y$. Let $e_\vartheta\in\omega_1\setminus e$ be such that $\psi(e_\vartheta)=y$. Then, for all $\eta<\vartheta$, $e_\eta<e\leq e_\vartheta$ and so, $e_\eta<e_\vartheta$. Moreover, $\psi(e_\eta)<y=\psi(e_\vartheta)$.
 
 Finally, $E:=\{e_\xi\colon\xi<\omega_1\}$ is as needed.
\end{proof}

\begin{lemma}\label{lemma:2}
 Let $\kappa$ be an infinite cardinal number with cofinality greater than some cardinal $\mu$. If for some family of sets $\{E_\xi:\xi<\mu\}$, $\kappa=\bigcup_{\xi<\mu}E_\xi$, then there exists $\alpha<\mu$ such that $|E_\alpha|=\kappa$.
\end{lemma}

\begin{proof}
 If $\kappa=\cf(\kappa)$, then the lemma is a common equivalence of being a regular cardinal. If, on the other hand, $\kappa$ is a singular cardinal, then \cite[Lemma~3.10, p.~32]{jech} gives our result immediately.
\end{proof}

 Given two functions $f,g\in\omega^\omega$, we write $f<^*g$ if there exists $k\in\omega$ such that $f(n)<g(n)$, whenever $n\in\omega\setminus(k+1)$. In other words, $f<^*g$ means that the inequality $f(n)<g(n)$ holds for all but finitely many natural numbers. Similarly, we define the relation $f\leq^*g$ ({\sl $g$ dominates $f$} or {\sl $f$ is dominated by $g$}) when $f(n)\leq g(n)$ is true for all but finitely many non-negative integers.
 
 Both relations, $<^*$ and $\leq^*$, are transitive. $\leq^*$ is also reflexive and the relations $f\leq^*g$ and $g\leq^*f$ imply that $\{n<\omega\colon f(n)=g(n)\}$ is cofinite, a property often abbreviated by the symbol $f=^*g$. We like to emphasize the fact that $<^*$ is not the strict pre-order induced by $\leq^*$, as it is easily seen by taking two functions that differ in exactly one point, for example.
 
 We say that a set $S\subseteq\omega^\omega$ is {\sl bounded} if there is a function from $\omega$ into $\omega$ which dominates all members of $S$. Otherwise, $S$ will be called {\sl unbounded}. Clearly, $\omega^\omega$ is unbounded, so let's denote by $\mathfrak b$ the minimum size of an unbounded set. This cardinal number is sometimes called {\sl the bounding number}.

\begin{lemma}\label{lemma:b}
 For each countable set $S\subseteq\omega^\omega$ there is $f$, a strictly increasing function from $\omega$ into $\omega$, in such a way that $g<^*f$, for all $g\in S$.
\end{lemma}

\begin{proof}
 Start by fixing $\{f_\ell\colon\ell<\omega\}$, an enumeration of $S$ (possibly with repetitions). Then the function $f\colon\omega\longrightarrow\omega$ given by $$f(n):=\max_{\ell\leq k\leq n}(f_\ell(k)+1)+n$$ is strictly increasing and satisfies that, for all $\ell<\omega$ and for all $k\geq\ell$, $f_\ell(k)<f(k)$.
\end{proof}

 The following inequalities hold as an immediate consequence of the previous result.

\begin{corollary}
 $\aleph_0<\frak b\leq\frak c.$
\end{corollary}

\section{The Open Coloring Axiom}

 This section is concerned with the axiom introduced by Todorc\v{e}vi\'{c} in \cite{stevo}. We start with a definition needed to state it.
 
\begin{definition}\label{def:color}
 Given a topological space $X$ and a set $K_0\subseteq[X]^2$, we let \[K_0^\sharp:=\{(x,y)\colon\{x,y\}\in K_0\}.\] Also, $K_0$ will be called {\sl open in $X$} if $K_0^\sharp$ is an open subset of the topological product $X\times X$.
 
 We can think of the set $[X]^2$ as the complete graph on the vertex set $X$ and also $K_0$ as a coloring of the edges of $X$ in some color $0$ and its complement in the color $1$. With this image in mind, a set $p\subseteq X$ is called {\sl 0-homogeneous} if $[p]^2\subseteq K_0$ and, similarly, {\sl 1-homogeneous} if $[p]^2\cap K_0=\emptyset$.
\end{definition}

 The following basic result concerning the previous definition will come in handy later in this work.
 
\begin{lemma}\label{lemma:3}
 If $X\subseteq\mathbb R$ and $K_0\subseteq[X]^2$ is open in $X$, then
\begin{enumerate}
 \item for all $v\in X$, $$\{x\in X:\{v,x\}\in K_0\}$$ is an open subset of $X$; and
 \item a set $A\subseteq X$ is $1$-homogeneous if and only if $\cl_X(A)$ is.
\end{enumerate}
\end{lemma}

\begin{proof}
\begin{enumerate}
 \item Routine arguments show that the natural projection $\pi\colon\{v\}\times X\to X$ given by $\pi(v,x):=x$ is a homeomorphism, in particular an open map. Being $K_0$ open in $X$, we have that $$\{x\in X:\{v,x\}\in K_0\}=\{x\in X:(v,x)\in K_0^\sharp\}=\pi[K^\sharp\cap(\{v\}\times X)]$$ is an open set.
 \item By counterpositive, assume $\cl_X(A)$ is not $1$-homogeneous, that is, there are $x_0,x_1\in\cl_X(A)$ such that $\{x_0,x_1\}\in K_0$. Since $K_0$ is open in $X$, we can find two open neighbourhoods $V_0,V_1$ of $x_0$ and $x_1$, respectively, such that $V_0\times V_1\subseteq K_0^\sharp$. But then, by elementary arguments of closure, there exists, for each $i<2$, $y_i\in V_i\cap A$. Hence, $\{y_0,y_1\}\in K_0$, and so $A$ cannot be $1$-homogeneous.
 
 The other implication is immediate since $A\subseteq\cl_X(A)$.
\end{enumerate} 
\end{proof}

\begin{definition}\label{def:oca}
 Given a topological space $Z$, OCA($Z$) is the statement: If $X$ is a subspace of $Z$ and $K_0\subseteq[X]^2$ is open in $X$, then there is either an uncountable 0-homogeneous set $Y\subseteq X$, or there exists $\{H_n\colon n<\omega\}$ such that $X=\bigcup_{n<\omega}H_n$ and, for all $n<\omega$, $[H_n]^2\subseteq[X]^2\setminus K_0$, in other words, $X$ is the union of countably many 1-homogeneous sets.
\end{definition}

 From now on, the Open Coloring Axiom (OCA) will be the statement OCA($\mathbb R$) (the reals with their usual topology). We now proceed to extend the axiom to a broader family of spaces.
 
\begin{lemma}\label{lemma:extend_oca}
 OCA($\mathbb R$) implies OCA($Z$) for any space $Z$ which can be embedded in $\mathbb R$.
\end{lemma}

\begin{proof}
 Assume OCA($\mathbb R$). Fix $X\subseteq Z$ and $K_0,K_1\subseteq [X]^2$ in such a way that $K_0$ is open in $X$ and $[X]^2=K_0\cupdot K_1$.
 
 If $h\colon Z\longrightarrow\mathbb R$ is a topological embedding, then $H:=h\times h$ is an embedding of $Z\times Z$ in $\mathbb R\times\mathbb R$. For $i<2$, we define $L_i:=\{h[a]\colon a\in K_i\}$. Since $[X]^2=K_0\cupdot K_1$, we obtain that $[h[X]]^2=L_0\cupdot L_1$.
 
 We want $L_0$ to be open in $h[X]$, which, according to our definition, means we want for $L_0^\sharp$ to be open in $h[X]\times h[X]$. But notice that $H[X\times X]=h[X]\times h[X]$ and the direct image of the open set $K_0^\sharp$ under $H$ is precisely $L_0^\sharp$. Since $H$ is an embedding, $L_0^\sharp$ must in fact be open, which means that, by OCA, we have two possible cases.
 
 First, there is an uncountable $Y\subseteq h[X]$ such that $[Y]^2\subseteq L_0$, which, by the bijectivity of $h$, implies that $h^{-1}[Y]$ is uncountable and $[h[Y]]^2\subseteq K_0$. Or, second, we can write $h[X]=\bigcup_{n<\omega}A_n$, where each set $A_n$ satisfies $[A_n]^2\subseteq L_1$. In this situation, it suffices to observe that $X=\bigcup_{n<\omega}h^{-1}[A_n]$ and clearly $[h^{-1}[A_n]]^2\subseteq K_1$, for every $n<\omega$.
\end{proof}

\section{OCA decides $\mathfrak b$}\label{section:oca_d}

 One of the first historic applications of OCA is that if we assume this axiom we can prove the equality $\frak b=\aleph_2$. For this, we use Lemma~\ref{lemma:extend_oca} and the fact that the topological product $\omega^\omega$ is homeomorphic to the irrationals endowed with the relative topology from $\mathbb R$, as proved in \cite[Teorema~1.17, p.~11]{dania}.
 
 We will abbreviate $X\times X$ as $X^2$ to save on notation.
 
\begin{lemma}\label{lemma:aleph1<b}
 Under OCA, every subset of $\omega^\omega$ of size $\aleph_1$ is bounded.
\end{lemma}

\begin{proof}
 Let $X_0\in[\omega^\omega]^{\aleph_1}$ be arbitrary. Enumerate $X_0=\{e_\alpha\colon\alpha<\omega_1\}$. For the sake of simplicity, we will construct a set $X=\{f_\alpha\colon\alpha<\omega_1\}\subseteq\omega^\omega$ in such a way that, for any $\alpha<\beta<\omega_1$,
\begin{enumerate}[label=\alph*)]
 \item $e_\alpha <^*f_\alpha$,
 \item $f_\alpha <^*f_\beta$, and
 \item $f_\alpha$ is strictly increasing.
\end{enumerate}

 By (a) and the transitivity of $<^*$, if we show that $X$ is bounded, then $X_0$ would be bounded as well and so we will be done.
 
 To construct $X$ we proceed by transfinite recursion: assume that for some $\beta<\omega_1$, $\{f_\alpha\colon\alpha<\beta\}$ has been constructed in such a way that (a), (b) and (c) hold. Notice that this family, even when we add $e_\beta$, is countable, so by Lemma~\ref{lemma:b}, there's a strictly increasing function $f_\beta$ such that for all $\alpha<\beta$, $f_\alpha<^*f_\beta$ and $e_\beta<^*f_\beta$. This completes our recursion.
 
 We assert that, under the assumption of OCA, $X$ is indeed bounded. Note that, according to (b), $|X|=\aleph_1$.
 
 Define $$K_0:=\left\{\{f_\alpha,f_\beta\}\colon \alpha<\beta\land\exists k<\omega\,(f_\alpha(k)>f_\beta(k))\right\}$$ and $K_1$ as its complement in $[X]^2$. Let us show that $K_0$ is open in $X$, in other words, that $K_0^\sharp$ is open in $X^2$.
 
 Take $\alpha<\beta<\omega_1$ and $k<\omega$ such that $f_\beta(k)<f_\alpha(k)$. By item (b) above, we have that there must be $m<\omega$ such that, for all $i\geq m$, $f_\alpha(i)<f_\beta(i)$ (in particular, $k<m$). Then, the basic open sets $U_\alpha:=[f_\alpha\restriction(m+1)]$ and $U_\beta:=[f_\beta\restriction(m+1)]$ clearly contain $f_\alpha$ and $f_\beta$, respectively.
 
 To verify that $(U_\alpha\times U_\beta)\cap X^2\subseteq K_0^\sharp$, consider $(p,q)\in(U_\alpha\times U_\beta)\cap X^2$. Then, for some $\xi,\eta<\omega_1$, $p=f_\xi$ and $q=f_\eta$. If $\xi<\eta$, then use the facts $k<m$ and $(f_\xi,f_\eta)\in U_\alpha\times U_\beta$ to deduce that $f_\eta(k)=f_\beta(k)<f_\alpha(k)=f_\xi(k)$; as a consequence, $\{f_\eta,f_\xi\}\in K_0$, which implies $(p,q)\in K_0^\sharp$. When $\eta<\xi$ one gets $f_\eta(m)=f_\beta(m)>f_\alpha(m)=f_\xi(m)$, and hence $\{f_\xi,f_\eta\}\in K_0$; therefore, $(p,q)\in K_0^\sharp$.
 
 In a similar fashion, $(f_\beta,f_\alpha)\in(U_\beta\times U_\alpha)\cap X^2\subseteq K_0^\sharp$ and, as a consequence, $K_0^\sharp$ is, indeed, an open subset of $X^2$.
 
 Seeking a contradiction, assume that $X=\bigcup_{n<\omega}H_n$, where $[H_n]^2\subseteq K_1$ for every $n<\omega$. Since $|X|=\aleph_1$, there exists $n<\omega$ with $|H_n|=\aleph_1$. Take the uncountable set of indeces $E:=\{\alpha<\omega_1\colon f_\alpha\in H_n\}$ and define, for every $\alpha\in E$, $S_\alpha:=\{(m,k)\in\omega\times\omega\colon m\leq f_\alpha(k)\}$.
 
 Suppose that $\alpha,\beta\in E$ satisfy $\alpha<\beta$. As a consequence of condition (b) one gets $\{f_\alpha,f_\beta\}\in[H_n]^2\subseteq K_1$. Our definition of $K_0$ guarantees that $f_\beta(i)\leq f_\alpha(i)$, for all $i<\omega$, and consequently, $S_\alpha\subseteq S_\beta$. Moreover, since $f_\alpha<^*f_\beta$, for some $k<\omega$, we have $f_\alpha(k)<f_\beta(k)$, and so $(f_\beta(k),k)\in S_\beta\setminus S_\alpha$. This argument shows that $S_\alpha$ is a proper subset of $S_\beta$.
 
 To complete this part of our argument, use the fact that $E$ is an uncountable subset of $\omega_1$ to obtain $\varphi\colon\omega_1\to E$, an order isomorphism. According to the previous paragraph, $\{ S_{\varphi(\alpha+1)}\setminus S_{\varphi(\alpha)}\colon\alpha<\omega_1\}$ is an uncountable family of non-empty pairwise disjoint subsets of the countable set $\omega\times\omega$, an absurdity.
 
 By OCA, we must have $Y\in[X]^{\aleph_1}$ with $[Y]^2\subseteq K_0$. We assert that $Y$ is bounded.
 
 Since $I:=\{\alpha<\omega_1\colon f_\alpha\in Y\}$ is uncountable, there is an order isomorphism $r:\omega_1\longrightarrow I$. Hence, by letting $g_\alpha:=f_{r(\alpha)}$, for each $\alpha<\omega_1$, one obtains an enumeration $\{g_\alpha\colon\alpha<\omega_1\}$ of $Y$ in such a way that $g_\alpha<^*g_\beta$, whenever $\alpha<\beta<\omega_1$.
 
 For every $t\in\omega^{<\omega}$ such that $[t]\cap Y\neq\emptyset$, choose some $\alpha_t<\omega_1$ that satisfies $t\subseteq g_{\alpha_t}$. Thus, $$\gamma:=\sup\{\alpha_t\colon(t\in\omega^{<\omega})\land([t]\cap Y\neq\emptyset)\}+1<\omega_1.$$ Next, for each integer $n$, set $A_n:=\{\beta<\omega_1\colon\forall k\in\omega\setminus n\,(g_\gamma(k)<g_\beta(k))\}$. The way we enumerated $Y$ guarantees that $\omega_1\setminus\gamma\subseteq\bigcup_{n<\omega}A_n$ and therefore, for some $k_0<\omega$, $A_{k_0}$ is uncountable. Now, the map from $A_{k_0}$ into $\omega^{<\omega}$ given by $\beta\mapsto g_\beta\restriction k_0$ has an uncountable fiber $Z$. To summarize, we have found an uncountable set $Z\subseteq\omega_1\setminus\gamma$ such that
\begin{align}\label{obs:Z}
 \text{for every }\xi,\eta\in Z\text{ and }k\in\omega\setminus k_0,\quad g_\gamma(k)<g_\xi(k)\text{ and }g_\xi\restriction k_0=g_{\eta}\restriction k_0.
\end{align}

 What remains of this proof is to show that $\{g_\beta\colon\beta\in Z\}$ is bounded by some function $g$. Note that, were this the case, we could find, for each $\alpha<\omega_1$, some $\beta\in Z\setminus(\alpha+1)$ ($Z$ is uncountable) such that $f_\alpha<^*g_\beta\leq^*g$ and so $X$ would be bounded as well.

 For all $n\geq k_0$, we define the sets $G_n:=\{g_\beta(n)\colon\beta\in Z\}$. We will prove, by contradiction, that all $G_n$ are finite. Thus let $m\geq k_0$ be the smallest natural number such that $G_m$ is infinite. By the choice of $m$, we deduce that $A:=\{g_\beta\restriction m\colon\beta\in Z\}$ is finite and $B:=\{g_\beta\restriction(m+1)\colon\beta\in Z\}$ is infinite. Therefore, the function from $B$ to $A$ given by $s\mapsto s\restriction m$ has an infinite fiber, which implies we can find $t\in\omega^m$ and $W\subseteq Z$ such that for all $\beta\in W$, $g_\beta\restriction m=t$ (in particular, $g_\beta\in[t]\cap Y$ and so, $\alpha_t$ exists) and $\{g_\beta(m)\colon\beta\in W\}$ is infinite.
 
 By our choice of $\gamma$, $\alpha:=\alpha_t<\gamma$ and thus $g_\alpha<^*g_\gamma$. Then there's some $k_1\geq m$ such that for all $k\geq k_1$, $g_\alpha(k)<g_\gamma(k)$. We can then pick $\beta\in W$ such that $g_\beta(m)\geq g_\gamma(k_1)$. Since $\beta\in W\subseteq Z\subseteq\omega_1\setminus\gamma$, $\alpha<\gamma\leq\beta$. Now, the fact that $r$ is an isomorphism implies that $\xi:=r(\alpha)<r(\beta)=:\eta$ and $\{g_\alpha,g_\beta\}=\{f_\xi,f_\eta\}\in[Y]^2\subseteq K_0$. Then there must be some $k<\omega$ such that $g_\beta(k)<g_\alpha(k)$.
 
 If $k<m$, then our choice of $t$, the equality $\alpha=\alpha_t$ and $\beta\in W$ imply that $g_\alpha(k)=t(k)=g_\beta(k)$, a contradiction. So, $m\leq k$. If $k\leq k_1$, by the way we picked $k_1$, and the fact that both $g_\alpha$ and $g_\beta$ are strictly increasing, we obtain $g_\alpha(k)\leq g_\alpha(k_1)<g_\gamma(k_1)\leq g_\beta(m)\leq g_\beta(k)$, another absurdity. Thus $k_1<k$. Since $k_0\leq m\leq k_1<k$ and $\beta\in W\subseteq Z$, (\ref{obs:Z}) implies that $g_\beta(k)<g_\alpha(k)<g_\gamma(k)<g_\beta(k)$. Once more, a contradiction.

 Therefore, all sets $G_n$ must be finite, and then $\{g_\beta\colon\beta\in Z\}$ is bounded by the function $$n\mapsto\begin{cases}0,&n<k_0\\\max G_n+1,&n\geq k_0.\end{cases}$$
\end{proof}

 We continue the section by introducing the concept of {\sl gaps in $\omega^\omega$} and some of their basic properties.

\begin{definition}
 Let $\kappa$ and $\lambda$ be regular cardinals. A {\sl $(\kappa,\lambda)$-gap in $\omega^\omega$} is a pair of sequences $\{f_\alpha\colon\alpha<\kappa\}\subseteq\omega^\omega$ and $\{g_\beta\colon\beta<\lambda\}\subseteq\omega^\omega$ such that for all $\alpha_0<\alpha_1<\kappa$ and $\beta_0<\beta_1<\lambda$, $$f_{\alpha_0}<^*f_{\alpha_1}<^*g_{\beta_1}<^*g_{\beta_0};$$ and there is no $h\in\omega^\omega$ satisfying that for all $\alpha<\kappa$ and $\beta<\lambda$, $f_\alpha<^*h<^*g_\beta$.
\end{definition}

 Intuitively, regarding the previous definition, the first sequence is strictly increasing and always below the second which is strictly decreasing; but the pair of sequences is ``tight'' in the sense one cannot fit a function that lies in between them, with respect to the strict pre-order $<^*$. Also, the restriction of both cardinals being regular is mostly inessential, since if we remove this restriction, every $(\kappa,\lambda)$-gap contains a $(\cf(\kappa),\cf(\lambda))$-gap that is ``cofinal.'' This observation is actually very useful. In most of the proofs that follow, the general idea is to construct such a cofinal gap.

 One very basic result concerning the existence of gaps can be stated as follows. Let us define the difference of two functions $f$ and $g$ in $\omega^\omega$ by $$(f-g)(n):=\max\{f(n)-g(n),0\}\text{ for every }n<\omega.$$

\begin{proposition}\label{prop:gaps_sim}
 If $\kappa$ and $\lambda$ are regular cardinals, then there is a $(\kappa,\lambda)$-gap if and only if there is a $(\lambda,\kappa)$-gap.
\end{proposition}

\begin{proof}
 For the first implication, consider $\{f_\alpha\colon\alpha<\kappa\}$ and $\{g_\beta\colon\beta<\lambda\}$ satisfying the definition of gap. Then for all $\beta_0<\beta_1<\lambda$ and $\alpha_0<\alpha_1<\kappa$, $$g_0-g_{\beta_0}<^*g_0-g_{\beta_1}<^*g_0-f_{\alpha_0}<^*g_0-f_{\alpha_1}.$$ Moreover, if there were a function $h$ such that for all $\beta<\lambda$ and $\alpha<\kappa$, $g_0-g_\beta<^*h<^*g_0-f_\alpha$, then $f_\alpha<^*g_0-h<^*g_\beta$, which would contradict our hypothesis. Hence, $\{g_0-g_\beta\colon\beta<\lambda\}$ and $\{g_0-f_\alpha\colon\alpha<\kappa\}$ form a $(\lambda,\kappa)$-gap.
 
 A routine argument shows that the remaining implication is also true.
\end{proof}

 Given any two arbitrary functions $f$ and $g$ in $\omega^\omega$, let us use the symbol $f<g$ when for all $n<\omega$, $f(n)<g(n)$; similarly, $f\leq g$ will be short for $f(n)\leq g(n)$ for every $n<\omega$.

\begin{lemma}\label{lemma:oca_gaps}
 Under OCA, there are no $(\kappa,\lambda)$-gaps in $\omega^\omega$, where $\kappa$ and $\lambda$ are regular uncountable cardinals and $\kappa>\omega_1$.
\end{lemma}

\begin{proof}
 Let $\kappa>\omega_1$ and $\lambda$ be regular uncountable cardinals. By Proposition~\ref{prop:gaps_sim}, we can assume, without losing generality, that $\kappa\geq\lambda$. We will first reach a contradiction by assuming that there is a gap of the form $\{f_\alpha\colon\alpha<\kappa\}$ and $\{g_\beta\colon\beta<\lambda\}$ which satisfies \begin{align}\label{obs:lambda}\text{for all }\alpha<\kappa,\text{ the set }S_\alpha:=\{\beta<\lambda\colon f_\alpha<g_\beta\}\text{ has size }\lambda.\end{align}
 
 Define $X:=\{(f_\alpha,g_\beta)\colon\alpha<\kappa\wedge\beta\in S_\alpha\}$. We think of $X$ as a subspace of $\omega^\omega\times\omega^\omega=(\omega^\omega)^2$. Since $(\omega^\omega)^2$ is homeomorphic to $\omega^{\omega\times 2}$, we apply Lemma~\ref{lemma:extend_oca} to deduce that OCA($X$) holds.
 
 If $x\in X$, denote $x=(x_0,x_1)$ for the sake of notational economy. Consider then the coloring given by $$K_0:=\{\{x,y\}\subseteq X\colon\exists k<\omega\,(x_0(k)>y_1(k)\vee x_1(k)<y_0(k))\}.$$ If we take $\{x,y\}\in K_0$ and fix $k<\omega$ as in the property that defines $K_0$, then, the open sets $U:=[x_0\restriction(k+1)]\times[x_1\restriction(k+1)]$ and $V:=[y_0\restriction(k+1)]\times[y_1\restriction(k+1)]$, by routine arguments, satisfy $(x,y)\in X^2\cap(U\times V)\subseteq K_0^\sharp$, proving that $K_0$ is open. By OCA, we have two possibilities which we discuss in detail.
 
 First, consider the case when $X=\bigcup_{n<\omega}H_n$, where for all $n<\omega$, $[H_n]^2\subseteq [X]^2\setminus K_0$. For every $\alpha<\kappa$, we define $\varphi_\alpha\colon S_\alpha\to\omega$ given by $$\varphi_\alpha(\beta):=\min\{n<\omega\colon(f_\alpha,g_\beta)\in H_n\}\text{ for each }\beta\in S_\alpha.$$ By the regularity of the uncountable cardinal $\lambda$, there exists $n_\alpha<\omega$ such that $T_\alpha:=\varphi_\alpha^{-1}[\{n_\alpha\}]$ has size $\lambda$.
 
 Similarly (recall that $\kappa$ is also a regular and uncountable cardinal), the function $\varphi\colon\kappa\to\omega$ given by $\varphi(\alpha):=n_\alpha$ has some fiber $A:=\varphi^{-1}[\{n\}]$ of size $\kappa$. By construction, for all $\alpha\in A$ and $\beta\in T_\alpha$, $(f_\alpha,g_\beta)\in H_n$.
 
 Next, fix $\gamma\in A$ and set $B:=T_\gamma$. We claim that for all $\alpha\in A$ and $\beta\in B$, $f_\alpha\leq g_\beta$. Indeed, pick $\delta\in T_\alpha\setminus\{\beta\}$ and notice that $\{(f_\alpha,g_\delta),(f_\gamma,g_\beta)\}\in[H_n]^2\subseteq[X]^2\setminus K_0$. By the definition of $K_0$, we have that for all $k<\omega$, $f_\alpha(k)\leq g_\beta(k)$ and $f_\gamma(k)\leq g_\delta(k)$. In particular, we obtain the desired inequality.
 
 Define $h\in\omega^\omega$ by $h(k):=\min\{g_\beta(k)\colon\beta\in B\}$ and observe that for all $\beta\in B$, $h\leq^*g_\beta$. Also, the claim proved in the previous paragraph guarantees that $f_\alpha\leq h$, whenever $\alpha\in A$. Since $A$ and $B$ have sizes $\kappa$ and $\lambda$, respectively, we deduce that, for any $\alpha<\kappa$ and $\beta<\lambda$, there exist $\xi\in A\setminus(\alpha+1)$ and $\eta\in B\setminus(\beta+1)$; therefore, $f_\alpha<^*f_\xi\leq h\leq g_\eta<^*g_\beta$, a contradiction to the definition of gap. This shows that our first case is impossible.
 
 Secondly, let us discard the case where there is $Y\subseteq X$ with $\aleph_1$ elements and such that $[Y]^2\subseteq K_0$. Clearly, $Y$ is a relation on $\omega^\omega$; moreover, we claim it is an injective function. Take $(f_\alpha,g_\beta),(f_\gamma,g_\delta)\in Y$ and suppose that $\beta<\delta$ or $\alpha<\gamma$. Either way, $\{(f_\alpha,g_\beta),(f_\gamma,g_\delta)\}\in[Y]^2\subseteq K_0$ and so there exists some natural number $k$ such that $g_\delta(k)<f_\alpha(k)$ or $g_\beta(k)<f_\gamma(k)$. In the former case, since $\delta\in S_\gamma$ and $\beta\in S_\alpha$, $f_\gamma(k)<g_\delta(k)<f_\alpha(k)<g_\beta(k)$; in the latter case, we use that $\beta\in S_\alpha$ and $\delta\in S_\gamma$ to obtain the inequalities $f_\alpha(k)<g_\beta(k)<f_\gamma(k)<g_\delta(k)$. In both cases, $f_\alpha\neq f_\gamma$ and $g_\beta\neq g_\delta$. Thus, $Y$ is indeed an injection.
 
 Fix $\{\gamma_\xi\colon\xi<\omega_1\}$ and  $\{\delta_\eta\colon\eta<\omega_1\}$, increasing enumerations of $\{\gamma<\kappa\colon f_\gamma\in\dom Y\}$ and $\{\delta<\lambda\colon g_\delta\in\img   Y\}$, respectively. Hence, $\ell:=\{(\xi,\eta)\in\omega_1\times\omega_1\colon(f_{\gamma_\xi},g_{\delta_\eta})\in Y\}$ is an injective function from $\omega_1$ onto $\omega_1$ and, according to Lemma~\ref{lemma:injection}, there is an uncountable set (which we enumerate in an increasing fashion) $E=\{\alpha_\xi\colon\xi<\omega_1\}$ in such a way that $\ell\restriction E$ is strictly increasing. Hence, by letting $\beta_\xi:=\delta_{\ell(\alpha_\xi)}$, for each $\xi<\omega_1$, we deduce that $\{\beta_\xi:\xi<\omega_1\}$ also forms an increasing sequence of $\aleph_1$ indices.
 
 Since $\kappa$ is regular and $\kappa>\omega_1$, $\delta:=\sup_{\vartheta<\omega_1}\alpha_\vartheta+1<\kappa$. If we consider the function $f_\delta$ we have that, for all $\vartheta<\omega_1$, $f_{\alpha_\vartheta}<^*f_\delta<^*g_{\beta_\vartheta}$. By arguments reminiscent of those we employed to construct $Z$ in the proof of Lemma~\ref{lemma:aleph1<b}, we obtain an uncountable set $Z\subseteq\omega_1$ and a natural number $m$ such that for all $\vartheta,\eta\in Z$ and $k\geq m$, $f_{\alpha_\vartheta}(k)<f_\delta(k)<g_{\beta_\eta}(k)$ and also $f_{\alpha_\vartheta}\restriction m=f_{\alpha_\eta}\restriction m$ and $g_{\beta_\vartheta}\restriction m=g_{\beta_\eta}\restriction m$.
 
 Fix two distinct ordinals $\vartheta,\eta\in Z$. The fact $Y\subseteq X$ implies that $\beta_\vartheta \in S_{\alpha_\vartheta}$ and so, $f_{\alpha_\vartheta}<g_{\beta_\vartheta}$. Now, if $k<m$, then $f_{\alpha_\vartheta}(k)<g_{\beta_\vartheta}(k)=g_{\beta_\eta}(k)$ and when $k\geq m$ we obtain $f_{\alpha_\vartheta}(k)<g_{\beta_\eta}(k)$. Hence, $f_{\alpha_\vartheta}<g_{\beta_\eta}$ and in a similar way, $f_{\alpha_\eta}<g_{\beta_\vartheta}$, a direct contradiction to $\{(f_{\alpha_\vartheta},g_{\beta_\vartheta}),(f_{\alpha_\eta},g_{\beta_\eta})\}\in K_0$. Thus we have proved that, under OCA, there are no gaps satisfying (\ref{obs:lambda}).
 
 The more general case follows by cleverly manipulating some arbitrary $(\kappa,\lambda)$-gap, say $\{d_\alpha\colon\alpha<\kappa\}$ and $\{e_\beta\colon\beta<\lambda\}$. For every $\alpha<\kappa$, by the regularity of the uncountable cardinal $\lambda$, there must exist some natural number $m_\alpha$ such that the set $\{\beta<\lambda\colon\forall k\geq m_\alpha\,(d_\alpha(k)<e_\beta(k))\}$ has size $\lambda$. Then, again following similar arguments, the function $\alpha\mapsto m_\alpha$ has a fiber, let's call it $A$, of size $\kappa$. In summary, for every $\alpha\in A$, we have that $$|\{\beta<\lambda\colon\forall k\geq m\,(d_\alpha(k)<e_\beta(k))\}|=\lambda.$$
 
 We then select some functions from our gap and shift them $m$ units. Formally, we let $s\in\omega^\omega$ be given by $s(n):=n+m$. If $\varphi\colon\kappa\to A$ is an order isomorphism and for every $\alpha<\kappa$ and $\beta<\lambda$ we write $f_\alpha:=d_{\varphi(\alpha)}\circ s$ and $g_\beta:=e_\beta\circ s$, we obtain a $(\kappa,\lambda)$-gap (recall that $|A|=\kappa$), $\{f_\alpha\colon\alpha<\kappa\}$ and $\{g_\beta\colon\beta<\lambda\}$, that satisfies condition (\ref{obs:lambda}). But we already proved that there are no such gaps.
\end{proof}

 For the next result of the section, we will use a theorem of Rothberger which appeared for the first time as \cite[Théorème~6, p.~121]{rothberger}, where, interestingly, he does not mention $\mathfrak b$ at all, and instead is concerned with proving an equivalence of the independent statement ``all families of functions in $\omega^\omega$ of order type $\omega_1$ are bounded.''
 
 Now we borrow an idea from \cite[Theorem~2.6, p.~37]{gaps}, where gaps are defined in terms of sets (and called {\sl non-separable gaps}), rather than functions. We do not need this definition, but we include it for the sake of completeness and for clarity of reasoning.
 
 Given two sets $A$ and $B$, $A\subset^*B$ means that $A\setminus B$ is finite, but $B\setminus A$ is infinite. Also, we will say that {\sl $A$ is almost contained in $B$} (and write $A\subseteq^*B$) whenever $A\setminus B$ is finite.
 
\begin{definition}\label{def:gaps2}
 For $\kappa$ and $\lambda$, two regular cardinals, a {\sl $(\kappa,\lambda)$-gap in $\mathcal P(\omega)$} is a pair of sequences $\{a_\alpha\colon\alpha<\kappa\}$ and $\{b_\beta\colon\beta<\lambda\}$, where for all $\alpha<\beta<\kappa$ and $\xi<\eta<\lambda$, $a_\alpha\cup b_\xi\subseteq\omega$ and $a_\alpha\subset^*a_\beta\subset^*b_\eta\subset^*b_\xi$. Moreover, if there is no $c\in\mathcal P(\omega)$ in such a way that $a_\alpha\subset^*c\subset^*b_\xi$ holds for every $\alpha<\kappa$ and $\xi<\lambda$, we will say that the gap is {\sl non-separable}.
\end{definition}

\begin{lemma}\label{lemma:kappa-gaps}
 If $\kappa<\mathfrak b$ is a cardinal, then there are no $(\kappa,\omega)$-gaps in $\omega^\omega$.
\end{lemma}

\begin{proof}
 For a function $f\in\omega^\omega$, we denote by $f^\downarrow$ the set of pairs of natural numbers $(m,n)$ such that $n\leq f(m)$, which is always an infinite subset of $\omega\times\omega$. It is straightforward to show that\begin{align}\label{obs:f_iff_a}f<^*g\text{ implies }f^\downarrow\subset^*g^\downarrow.\end{align}
  So, if for some cardinal $\kappa$ we have a $(\kappa,\omega)$-gap in $\omega^\omega$ consisting of $\{f_\alpha\colon\alpha<\kappa\}$ and $\{g_n\colon n<\omega\}$, we define, for all $\alpha<\kappa$ and $n<\omega$, $A_\alpha:=f_\alpha^\downarrow$ and $B_n:=g_n^\downarrow$.
 
 According to (\ref{obs:f_iff_a}), $\{A_\alpha\colon\alpha<\kappa\}$, $\{B_n\colon n<\omega\}$ is a $(\kappa,\omega)$-gap in $\omega\times\omega$. To verify that it is non-separable, we prove the following claim.
 
\begin{nclaim}
 There is no set $S\subseteq\omega\times\omega$ such that for all $\alpha<\kappa$ and $n<\omega$, $A_\alpha\subset^*S\subset^*B_n$.
\end{nclaim}

\begin{claimproof}
 By contradiction, suppose there is such a set $S$. Next, define $V_n:=\{n\}\times\omega$ and observe that $$S\cap V_n=(S\cap V_n\cap B_0)\cup(S\cap V_n\setminus B_0).$$ Since $V_n\cap B_0$ and $S\setminus B_0$ are both finite, $S\cap V_n$ is finite as well.
 
 Consequently, we can define $h(n):=\max(\{0\}\cup\{i<\omega\colon(n,i)\in S\})$ which yields a function $h\in\omega^\omega$ such that $S\subseteq h^\downarrow$ and, as a consequence, $f_\alpha<^*h$ for all $\alpha<\kappa$. Following rudimentary methods, for all $n<\omega$ there exists $m<\omega$ such that $$S\setminus \bigcup_{i=0}^mV_i\subseteq g_{n+1}^\downarrow$$ and as such, $h\leq^*g_{n+1}<^*g_n$. A contradiction to the fact that $\{f_\alpha\colon\alpha<\kappa\}$, $\{g_n\colon n<\omega\}$ is a gap in $\omega^\omega$.
\end{claimproof}
 
 Using the fact that $\kappa<\mathfrak b$, one can easily contradict the Claim: by choosing a function $f$ that dominates all functions of the form $$h_\alpha(n):=\max(\{0\}\cup\{j<\omega:(n,j)\in A_\alpha\cap B_n\}),$$ for $\alpha<\kappa$, we can consider $S:=f^\downarrow$. Now, given $\alpha<\kappa$ and $n<\omega$, the hypothesis $A_\alpha\subseteq^* B_n$, together with (\ref{obs:f_iff_a}), imply that $A_\alpha\setminus S$ is almost contained in the finite set $(A_\alpha\cap B_n)\setminus h_\alpha^\downarrow$. Employing similar arguments, it is verified that $S\subset^*B_n$, and so we are done.
\end{proof}

\begin{lemma}\label{lemma:kappa<b}
 Let $\kappa\leq\mathfrak b$ be a cardinal number. Then there exists $\{f_\alpha\colon\alpha<\kappa\}\subseteq\omega^\omega$ such that for all $\alpha<\beta<\kappa$,
\begin{enumerate}
 \item $f_\alpha<^*f_\beta$ and
 \item $f_\alpha$ is strictly increasing.
\end{enumerate} 
\end{lemma}

\begin{proof}
 Suppose that for some $\alpha<\kappa$, $\{f_\xi\colon\xi<\alpha\}$ satisfies (1) and (2). Since $\alpha<\mathfrak b$ and then by Lemma~\ref{lemma:b}, there is a strictly increasing function $f_\alpha$ such that for all $\xi<\lambda$, $f_\xi<^*f_\alpha$.
\end{proof}

\begin{theorem}\label{theo:b=omega_2}
 OCA implies that $\mathfrak b=\aleph_2$.
\end{theorem}

\begin{proof}
 By assuming that $\aleph_2<\mathfrak b$, we construct an $(\aleph_2,\lambda)$-gap where $\lambda$ is a regular uncountable cardinal and thus, by Lemmas \ref{lemma:aleph1<b} and \ref{lemma:oca_gaps} and, of course, OCA, we complete the proof.
 
 Since $\aleph_2<\mathfrak b$, Lemma~\ref{lemma:kappa<b} yields a sequence $\{f_\alpha\colon\alpha<\aleph_2\}$ as stated in the lemma. Moreover, the same inequality produces a function $f\in\omega^\omega$ which dominates each $f_\alpha$.
 
 Let $U:=\{h\in\omega^\omega\colon\forall\alpha<\omega_2\,(f_\alpha<^*h)\}$. Now define $\mathbb P$ by the formula, $p\in\mathbb P$ if and only if there exists an ordinal number $\delta>0$ such that $p\colon\delta\longrightarrow U$ and for all $\xi<\eta<\delta$, $p(\eta)<^*p(\xi)$. 
 
 Clearly, $\{(0,f)\}$ is a member of $U$ and thus $(\mathbb P,\subseteq)$ is a nonempty partial order. If $\mathcal C$ is any chain in said partial order, clearly $f:=\bigcup\mathcal C$ is a function whose domain is the ordinal $\delta:=\sup\{\dom(p)\colon p\in\mathcal C\}$ and so, $f\colon\delta\to U$. Furthermore, given $\xi<\eta<\delta$, there exist $p,q\in\mathcal C$ such that $\xi\in\dom(p)$ and $\eta\in\dom(q)$. Then, since $\mathcal C$ is a chain, $r:=p\cup q\in\{p,q\}\subseteq\mathbb P$ and, as such, $f(\eta)=r(\eta)<^*r(\xi)=f(\xi)$. We proved that $f\in\mathbb P$.
 
 Let $p$ be a maximal element of $\mathbb P$ (we're using Zorn's Lemma here) and set $\delta:=\dom(p)$.

\begin{nclaim}
 There is no $h\in\omega^\omega$ satisfying $f_\alpha<^*h<^*p(\beta)$, for all $\beta<\delta$.
\end{nclaim}
 
 Indeed, if $h$ satisfies the inequalities given in the Claim, then $q:=p\cup\{(\delta,h)\}$ is a member of $\mathbb P$ such that $p\subsetneq q$, a contradiction to $p$'s maximality.
 
 Let us now use the Claim to prove that $\delta$ is a limit ordinal. Seeking a contradiction, suppose $\delta=\gamma+1$ and define $h\in\omega^\omega$ by $h(n):=\max\{p(\gamma)(n)-1,0\}$, for each $n<\omega$. The fact that $f_0<^*p(\gamma)$ implies that, for some $k<\omega$ and all $n\in\omega\setminus k$, $f_0(n)<p(\gamma)(n)$, i.e., $h(n)=p(\gamma)(n)-1$, whenever $n\in\omega\setminus k$. As a consequence, for each $\beta<\delta$, $h<^*p(\beta)$. On the other hand, given $\alpha<\omega_2$, we obtain $f_\alpha<^*f_{\alpha+1}<^*p(\gamma)$ and therefore, $f_\alpha<^*h$; a flagrant contradiction to our Claim.
 
 From the previous paragraph we conclude that $\lambda:=\cf(\delta)$ is an infinite cardinal. Let $\{\delta_\beta\colon\beta<\lambda\}$ be a cofinal subset of $\delta$ with $\delta_\beta<\delta_\gamma$, for all $\beta<\gamma<\lambda$. Then, according to our Claim, the pair $\{f_\alpha\colon\alpha<\omega_2\}$, $\{p(\delta_\beta)\colon\beta<\lambda\}$ is an $(\aleph_2,\lambda)$-gap. Hence, Lemma~\ref{lemma:kappa-gaps} implies that $\lambda>\omega$ and this completes our proof.
\end{proof}

\section{A very brief introduction to PFA}
 
 We now concentrate on the consistency of OCA and its negation with ZFC. The previous section is dedicated to proving that, under OCA, $\frak b=\aleph_2$. Given that in ZFC $\frak b\leq\mathfrak c$, the continuum hypothesis is clearly inconsistent with OCA, meaning that in any model in which CH is true, OCA is necessarily false and thus, by the consistency of CH, it is consistent with ZFC that OCA is false. All that remains is to exhibit a model of ZFC where OCA holds and we will have proven the relative consistency of Todorc\v{e}vi\'{c}'s axiom.
 
 The main focus of the final sections is showing that the Proper Forcing Axiom, commonly abbreviated as PFA, implies OCA. Showing that PFA is itself consistent involves assuming the existence of a supercompact cardinal to produce a generic model of ZFC in which PFA is true. For more details on the consistency of PFA we direct the interested reader to \cite[Chapter~31, p.~607]{jech}.
 
 With the hope of making the present paper accessible to more readers, our presentation of PFA will be through infinitary games (thus, avoiding the use of elementary submodels or preservation of stationary sets in generic extensions).
 
 We use the traditional notation concerning {\sl forcing}, which can be found in \cite[Chapter~VII, p.~184]{kunen}.
 
\begin{definition}\label{def:game}
 Given a forcing notion $\mathbb P$ and some element $p\in\mathbb P$, the {\sl proper game under $p$ in $\mathbb P$} is a turn-based two-player game $\Game(\mathbb P,p)$ with the following rules. Given $n<\omega$, in his $n$th turn, {\tt Player I} chooses a $\mathbb P$-name $\dot\alpha_n$ such that $p\Vdash\dot\alpha_n\in\ON$ (to be read as ``$\dot\alpha_n$ is an ordinal number''). In response, {\tt Player II} chooses an ordinal number $\beta_n$.
 
\begin{center}
\begin{tabular}{c|c c c c} 
 {\tt Player I} & $\dot\alpha_0$ & $\dot\alpha_1$ & \dots & $\dot\alpha_n$ \\ 
 \hline
 {\tt Player II} & $\beta_0$ & $\beta_1$ & \dots & $\beta_n$
\end{tabular}
\end{center}
 
 After $\omega$ turns, the game will have produced two sequences, $\{\dot\alpha_n:n<\omega\}$ and $\{\beta_n:n<\omega\}$. {\tt Player II} wins the game if there is some $q\leq p$ (naturally, $\leq$ is the corresponding order of the forcing notion $\mathbb P$) such that $$q\Vdash\forall n<\omega\ \exists k<\omega\ (\dot\alpha_n=\beta_k);$$ otherwise, {\tt Player I} wins.
\end{definition}

 Hence, a forcing notion $\mathbb P$ is said to be {\sl proper} if for every $p\in\mathbb P$, {\tt Player II} has a winning strategy for the game $\Game(\mathbb P,p)$.

\begin{definition}[Proper Forcing Axiom]\label{def:pfa}
 Let $\mathbb P$ be a proper forcing notion. If $\mathcal D$ is a family of no more than $\aleph_1$ dense subsets of $\mathbb P$, then there is a $\mathcal D$-generic filter, that is, a filter in $\mathbb P$ that intersects every member of $\mathcal D$.
\end{definition}

 The following result is a straightforward consequence of \cite[Theorem~31.15, p.~604]{jech}. We refer the reader to \cite[Chapter~VIII, p.~251]{kunen} for details related to two-step iterations.

\begin{lemma}[Shelah]\label{lemma:shelah}
 If $P$ is a proper forcing notion and $\dot Q$ is a $P$-name such that $$1\Vdash\dot Q\text{ is proper},$$ then the two-step iteration $P*\dot Q$ is also proper.
\end{lemma}

 Recall that a forcing notion is {\sl $\sigma$-closed} if every countable descending chain of elements has a lower bound (i.e., there exists some element that is below all members of the chain). A forcing notion is called {\sl ccc} (or said to have the {\sl countable chain condition}) if all its antichains are countable.
 
 As a way of better illustrating the notion of proper, we show that all $\sigma$-closed and all ccc forcing notions are proper.
 
\begin{lemma}\label{lemma:sigma}
 Let $\mathbb P$ be a forcing notion. If $\mathbb P$ is $\sigma$-closed, then $\mathbb P$ is proper.
\end{lemma}

\begin{proof}
 Let $p$ be an arbitrary element of a $\sigma$-closed forcing notion $\mathbb P$. We need to come up with a winning strategy for {\tt Player II}. By the inductive principle, assume it's the $n$th turn and we have,
\begin{enumerate}
 \item the first $n+1$ plays of {\tt Player I}, $\{\dot\alpha_i:i\leq n\}$;
 \item {\tt Player II}'s responses $\{\beta_i:i<n\}$; and
 \item a sequence $\{q_i:i<n\}\subseteq\mathbb P$ in such a way that $q_i\leq p$ and $q_i\Vdash\dot\alpha_i=\beta_i$; moreover, whenever $i+1<n$, $q_{i+1}\leq q_i$.
\end{enumerate}
 Hence, we fix $q_n\in\mathbb P$, a lower bound of $\{q_i:i<n\}$, and an ordinal number $\beta_n$ with $q_n\Vdash\dot\alpha_n=\beta_n$. Naturally, $\beta_n$ is the $n$th play of {\tt Player II} in $\Game(\mathbb P,p)$. Let's argue now that if the second player follows this strategy, he wins the match.
 
 Using the hypothesis that $\mathbb P$ is $\sigma$-closed, we can find a $q\in\mathbb P$ such that for all $n<\omega$, $q\leq q_n$. In particular, $q\leq p$ and $q\Vdash\dot\alpha_n=\beta_n$ for every $n$, and hence we are done.
\end{proof}

 Assume that $\mathbb P$ is a ccc forcing notion. We claim that if $p\in\mathbb P$ and $\dot\alpha$ is a $\mathbb P$-name with $p\Vdash\dot\alpha\in\ON$, then there is $B$, a countable set of ordinals, with $p\Vdash\dot\alpha\in B$ (in other words, there are only countably many different values for $\dot\alpha$). Indeed, let's begin by noticing that $$D:=\{q\in\mathbb P:\exists\beta\in\ON\ (q\Vdash\dot\alpha=\beta)\}$$ is dense below $p$. If we pick a maximal antichain in $D$, say $A$, then our hypothesis on $\mathbb P$ implies that $A$ is countable. Notice that for every $q\in A$, there exists $\beta_q\in\ON$ such that $q\Vdash\dot\alpha=\beta_q$. Being $A$ countable, immediately we have that $B:=\{\beta_q:q\in A\}$ is a countable set of ordinal numbers. In order to prove that $p\Vdash\dot\alpha\in B$, we will argue that $\{r\in\mathbb P:r\Vdash\dot\alpha\in B\}$ is dense below $p$.
 
 With this in mind, notice that for every $t\leq p$ there corresponds a $q\in\mathbb P$ and an ordinal number $\beta$ such that $q\leq t$ and $q\Vdash\dot\alpha=\beta$. By the maximality of $A$ in $D$, we have that some element $u\in A$ must be compatible with $q$. Let $t\leq q,p$ be a witness of this compatibility. Since $s\leq u\in A$, $s$ also forces that $\dot\alpha=\beta_u\in B$.

\begin{lemma}
 Every ccc forcing notion is proper.
\end{lemma}

\begin{proof}
 For every $n<\omega$, define the vertical line $V_n:=\{n\}\times\omega$ and $W_n:=\bigcup\{V_k:k\leq n\}$. We now fix a bijection $f\colon\omega\to\omega\times\omega$ that fulfills the condition
\begin{align}\label{eq:star2}
 \text{for all }n<\omega,\ f(n)\in W_n.
\end{align}

 Let $p\in\mathbb P$ and $n<\omega$. Similar to the proof of Lemma~\ref{lemma:sigma}, suppose we are currently at the $n$th turn and we already have the following.
\begin{enumerate}
 \item {\tt Player I} has played $\{\dot\alpha_k:k\leq n\}$;
 \item for all $k<n$, the collection of ordinals $\{\beta_i:i\in W_k\}$ satisfies $$p\Vdash\exists i\in W_k\ (\dot\alpha=\beta_i);$$ and
 \item {\tt Player II} has responded with $\{\beta_{f(k)}:k<n\}$.
\end{enumerate}

\begin{center}
\begin{tabular}{c|c c c c} 
 {\tt Player I} & $\dot\alpha_0$ & $\dot\alpha_1$ & \dots & $\dot\alpha_n$ \\ 
 \hline
 {\tt Player II} & $\beta_{f(0)}$ & $\beta_{f(1)}$ & \dots & \framebox(15,6){} 
\end{tabular}
\end{center}

 We need to choose the move for {\tt Player II}. By the remark preceding this lemma and the observation that $p\Vdash\dot\alpha_n\in\ON$, there is some countable set of ordinals $B_n$ such that $p\Vdash\dot\alpha_n\in B_n$. Since $V_n$ has size $\aleph_0$, we can use it as a set of indices to enumerate (possibly with repetitions) $B_n$, i.e., $B_n=\{\beta_i:i\in V_n\}$. Then we apply (\ref{eq:star2}) to acquire the ordinal number $\beta_{f(n)}$ and cast it as {\tt Player II}'s $n$th play.

 Let's argue that the strategy described above is a winning strategy for the second player. Start by setting $q:=p$ to obtain $q\leq p$. On the other hand, $f$'s surjectivity, together with (2), implies that $p\Vdash\exists k<\omega\ (\dot\alpha_n=\beta_{f(k)})$. Finally, this strategy guarantees {\tt Player II} comes out as victor.
\end{proof}

\section{Todorc\v{e}vi\'{c}'s Lemma}

 This section is entirely dedicated to proving a very important lemma needed to show that PFA implies OCA in the next section. We begin by introducing some notation.
 
\begin{definition}
 For any $Y\subseteq\mathbb R$ and $K\subseteq[Y]^2$, we define the set $$\mathbb P_Y(K):=\{p\in[Y]^{<\omega}:[p]^2\subseteq K\}.$$ In other words, $\mathbb P_Y(K)$ consists of all finite subsets of $Y$ which are homogeneous with respect to the coloring $K$.
\end{definition}

 Observe that $\mathbb P_Y(K)$ forms a nonempty (clearly $\emptyset\in\mathbb P_Y(K)$) partial order with the reverse inclusion.
 
 The rest of the present section is devoted to the proof of Lemma~\ref{lemma:todorcevic} and having that in mind, we will use the symbol $\mathbb N$ to denote the set $\omega\setminus\{0\}$. Hence, given $n\in\mathbb N$ and a set $X$, the symbol $X^n$ will denote the collection of all $n$-tuples in $X$ and, at the same time, the collection of all functions from the ordinal $n$ into $X$. Naturally, when $X\subseteq\mathbb R$, $X^n$ will be considered as a subspace of the Euclidean space $\mathbb R^n$.
 
\begin{lemma}[Todorc\v{e}vi\'{c}]\label{lemma:todorcevic}
 Suppose that $X\subseteq\mathbb R$ and that $K_0\subseteq[X]^2$ is open in $X$. If $X$ is not the union of less than $\mathfrak c$ 1-homogeneous sets (recall the second paragraph of Definition~\ref{def:color}), then there is a set $Y\subseteq X$ of size $\mathfrak c$ such that for any pairwise disjoint family $\mathcal A\in[\mathbb P_Y(K_0)]^\mathfrak c$, there exist two distinct elements in $\mathcal A$ whose union is 0-homogeneous.
\end{lemma}

\begin{proof}
 Let $X$ and $K_0$ be like in the statement of the lemma. For any $n\in\mathbb N$ and $x\in X^n$, we define the following.
\begin{enumerate}
 \item $\mathcal O(x)$ is the family of open sets in $X^n$ containing $x$;
 \item for all $U\in\mathcal O(x)$, $$\mathcal W(x,U):=\{y\in U:\forall i,j<n\ \{x(i),y(j)\}\in K_0\};$$
 \item given $S\subseteq X^n$ and $g\colon S\to X$, the {\sl oscillation of $g$ at $x$} is defined as $$o(g,x):=\bigcap\left\{\cl_Xg[S\cap\mathcal W(x,U)]:U\in\mathcal O(x)\right\}.$$
\end{enumerate}

 Concerning these definitions, one direct observation is that whenever $S$ happens to be countable, $\left|X^S\right|\leq|X|^\omega\leq\mathfrak c^\omega=\mathfrak c$. On the other hand, $\left|[X^n]^{\leq\omega}\right|\leq\left|X^n\right|^\omega\leq\mathfrak c^\omega=\mathfrak c$. Both these observations yield that the collection of countable functions which are subsets of some product of the form $(X^n)\times X$ has at most $\omega\cdot\mathfrak c\cdot\mathfrak c=\mathfrak c$ elements. Hence, we can enumerate it as
 
\begin{align}\label{eq:enum_f}
 \bigcup_{n=1}^\infty\left\{X^S:s\in[X^n]^{\leq\omega}\right\}=\{f_\xi:\xi<\mathfrak c\}.
\end{align}
 
\begin{claim}\label{claim:Wisopen}
 The sets $\mathcal W(x,X^n)$ are all open in $X^n$.
\end{claim} 
 
\begin{claimproof}
 Take $x\in X^n$. According to Lemma~\ref{lemma:3}, the set $$U:=\bigcap_{j<n}\left\{y\in X:\{x(j),y\}\in K_0\right\}$$ is open in $X$. It is therefore sufficient to argue that $U^n=\mathcal W(x,X^n)$. We do this via a double inclusion.
 
 For every $y\in U^n\subseteq X^n$ and $i<n$, we have that $y(i)\in U$, that is, for every $i,j<n$, $y(i)\in U\subseteq\{z\in X:\{x(j),y(i)\}\in K_0\}$. Finally, $y\in\mathcal W(x,X^n)$.
 
 If we now take $y\in\mathcal W(x,X^n)$, we particularly have that $y\in X^n$ and that for all $i<n$, $y(i)\in U$. And thus $y\in U^n$.
\end{claimproof}
  
 Next we notice that, since $\mathbb R$ is a second countable topological space, $X$ must have, as a subspace of the reals, at most $\mathfrak c$ closed sets. In particular, we can enumerate all closed subsets of $X$ that happen to be 1-homogeneous as $\{H_\xi:\xi<\mathfrak c\}$.
 
\begin{claim}\label{claim:1}
 There is a set $\{y_\xi:\xi<\mathfrak c\}$ such that for all $\alpha<\mathfrak c$,
 \begin{enumerate}[label=(\roman*)]
  \item $y_\alpha\in X\setminus\{y_\xi:\xi<\alpha\}$;
  \item $y_\alpha\notin\bigcup\{H_\xi:\xi<\alpha\}$;
  \item for all $z\in\{y_\xi:\xi<\alpha\}^{<\omega}$ and $\eta<\alpha$, if for some positive integer $n$, $\dom(f_\eta)\subseteq X^n$, $\dom(z) = n$ (recall that we are identifying $n$-tuples with functions whose domain is the ordinal $n$), and $o(f_\eta,z)$ is 1-homogeneous, then $y_\alpha\notin o(f_\eta,z)$.
 \end{enumerate}
\end{claim}

\begin{claimproof}
 We will construct the set recursively. Assume we have already constructed $\{y_\xi:\xi<\alpha\}$, for some $\alpha<\mathfrak c$, with the aforementioned properties. First, notice that for all $\xi<\alpha$, the set $[\{y_\xi\}]^2=\emptyset$ is trivially disjoint from $K_0$; in other words, the collection $$\mathcal H_1:=\left\{\{y_\xi\}:\xi<\alpha\right\}$$ consists of 1-homogeneous sets. Let us also define $$\mathcal H_2:=\{H_\xi:\xi<\alpha\}$$ and $\mathcal H_3$ as the collection of all 1-homogeneous sets of the form $o(f_\eta,z)$, where $z\in\{y_\xi:\xi<\alpha\}^n$, $\eta<\alpha$, and $\dom(f_\eta)\subseteq X^n$.
 
 It is then straightforward to see that each of these three collections has fewer that $\mathfrak c$ elements, which, in turn, implies that there must be some $y_\alpha\in X\setminus\bigcup(\mathcal H_1\cup\mathcal H_2\cup\mathcal H_3)$. By the construction of the sets $\mathcal H_i$ for $i\in\{1,2,3\}$, $y_\alpha$ has the desired properties. This finishes the recursion.
\end{claimproof}

 We now claim that $Y:=\{y_\xi:\xi<\mathfrak c\}$ accomplishes what the lemma concludes. Moreover, we make the following rather insightful remark about this proof as a whole, which summarizes exactly what is important about the set $Y$.

\begin{remark}\label{remark:afirm3}
 Any set $\{y_\xi : \xi<\mathfrak c\}$ satisfying (i), (ii) and (iii) of Claim~\ref{claim:1} fulfills the conclusion of Lemma~\ref{lemma:todorcevic}.
\end{remark}
 
 Immediately, (i) of Claim~\ref{claim:1} implies that the size of $Y$ is precisely $\mathfrak c$. All that remains to conclude our argument is to take a pairwise disjoint family $\hat{\mathcal A}\subseteq\mathbb P_Y(K_0)$ of size continuum and conclude the existence of two distinct elements in the family whose union is 0-homogeneous. The following claim argues that this sought after conclusion holds if it does so for all families whose members are of the same size.

\begin{claim}\label{claim:n_enough}
 If $\hat{\mathcal A}\subseteq\mathbb P_Y(K_0)$ is pairwise disjoint and of cardinality continuum, then there exist $\mathcal A\subseteq\hat{\mathcal A}$ and $n\in\mathbb N$ with $|\mathcal A| = \mathfrak c$ and $\mathcal A\subseteq[Y]^n$.
\end{claim}

\begin{claimproof}
 Assuming that $\hat{\mathcal A}$ is as described in the Claim, define, for each $n<\omega$, $\mathcal A_n:=\{p\in\hat{\mathcal A}:|p|=n\}$ and observe that $$\hat{\mathcal A}=\bigcup_{n<\omega}\mathcal A_n.$$ Next, by König's Lemma, we know that $\mathfrak c$ has uncountable cofinality, which by Lemma~\ref{lemma:2} implies the existence of $n<\omega$ such that $|\mathcal A_n|=\mathfrak c$. Obviously $n$ must be a positive integer since $\mathcal A_0\subseteq\{\emptyset\}$. Finally, $\mathcal A:=\mathcal A_n$ satisfies the conclusion of the Claim.
\end{claimproof}

 With the notation used in the statement of Claim~\ref{claim:n_enough}, if $\mathcal A$ has two distinct members whose union is $0$-homogeneous, then $\hat{\mathcal A}$ has them too. In other words, we only need to verify by induction on $n$ that if $\mathcal A$ is a pairwise disjoint subset of $\mathbb P_Y(K_0)$ with $|\mathcal A| = \mathfrak c$ and $\mathcal A\subseteq[Y]^n$, then there are $p,q\in\mathcal A$ for which $p\neq q$ and $p\cup q$ is $0$-homogeneous.
 
 When $n=1$, we have that $A:=\bigcup\mathcal A$ has cardinality $\mathfrak c$. Our choice of $\mathcal A$ implies that for all $\alpha<\mathfrak c$, $A\cap H_\alpha$ is a subset of $Y\cap H_\alpha$, which in turn (see (ii) of Claim~\ref{claim:1}) is contained in $\{y_\xi:\xi\leq\alpha\}$, a set with fewer than $\mathfrak c$ elements. As a consequence, $A$ is not a subset of any $H_\alpha$. Thus, the relation $A\subseteq\cl_X(A)$ implies that $\cl_X(A)$ is a closed subset of $X$ which does not belong to $\{H_\xi : \xi<\mathfrak c\}$, in other words, the closure of $A$ in $X$ is not $1$-homogeneous.
 
 By Lemma~\ref{lemma:3}(2), $A$ fails to be $1$-homogeneous. Then, we can find two members of $A$, let's say $u$ and $v$, in such a way that $\{u,v\}\in K_0$. In conclusion, by letting $p:=\{u\}$ and $q:=\{v\}$ we get two distinct elements of $\mathcal A$ whose union is $0$-homogeneous. This completes the base of our finite induction.
 
 For the inductive step, suppose $\mathcal A\subseteq\mathbb P_Y(K_0)$ is a pairwise disjoint family of $\mathfrak c$ subsets of $Y$ all of size $n+1$. For every $p\in\mathcal A$, one gets $p\subseteq\{y_\xi : \xi<\mathfrak c\}$ and thus, there is a function $e(p)\colon n+1\to\mathfrak c$ satisfying the following conditions. 
\begin{enumerate}[label=(\Roman*)]
 \item $p=\{y_{e(p)(i)}:i\leq n\}$ and
 \item if $i<j\leq n$ then $e(p)(i)<e(p)(j)$ (in other words, $e(p)$ is strictly increasing).
\end{enumerate}
Moreover, we define $\vec p\colon n+1\to X$ by letting $\vec p(i):=y_{e(p)(i)}$. We have the following immediate observation.

\begin{align}\label{eq:star}
 \vec p\colon n+1\to p\text{ is a bijection}.
\end{align}

 Let $S:=\{\vec p\restriction n:p\in\mathcal A\}\subseteq X^n$. Given two different elements $p,q\in\mathcal A$, by our hypothesis on $\mathcal A$, $p\cap q=\emptyset$. Since also $\vec p(0)\in p$ and $\vec q(0)\in q$, clearly $\vec p\restriction n\neq\vec q\restriction n$ (recall that $n>0$). This observation allows us to define a function $g\colon S\to X$ by setting $g(\vec p\restriction n):=\vec p(n)$.
 
 By the above definitions, it makes sense to consider the family $$\mathcal A':=\{p\in\mathcal A:\vec p(n)\in o(g,\vec p\restriction n)\}.$$

\begin{claim}\label{claim:2}
 The set $\mathcal A\setminus\mathcal A'$ has (strictly) fewer than $\mathfrak c$ elements.
\end{claim}

\begin{claimproof}
 By the second countable property of the reals, we begin by taking a countable base $\mathcal B$ for $X$ and, similarly, a countable base $\mathcal B^*$ for $X^n$.
 
 Fix an arbitrary element $p\in\mathcal A\setminus\mathcal A'$, that is, $p\in\mathcal A$ and $\vec p(n)\notin o(g,\vec p\restriction n)$. By the definition of oscillation, we can find $U\in\mathcal O(\vec p\restriction n)$ and $B_p\in\mathcal B$ such that 
 
\begin{align}\label{eq:diam}
 \vec p(n)\in B_p\subseteq X\setminus\cl_X g[S\cap\mathcal W(\vec p\restriction n,U)].
\end{align}

\noindent Hence, there exists $B_p^*\in\mathcal B^*$ with the property that $\vec p\restriction n\in B_p^*\subseteq U$. This last inclusion guarantees that $\mathcal W(\vec p\restriction n,B_p^*)\subseteq\mathcal W(\vec p\restriction n,U)$ and so, by using (\ref{eq:diam}) and some routine arguments, we obtain that
 
\begin{align*}
 B_p\subseteq X\setminus\cl_X g[S\cap\mathcal W(\vec p\restriction n,B_p^*)]\subseteq X\setminus g[S\cap\mathcal W(\vec p\restriction n,B_p^*)].
\end{align*}

 \noindent Therefore,

\begin{align}\label{eq:m}
 g[S\cap\mathcal W(\vec p\restriction n,B_p^*)]&\subseteq X\setminus B_p.
\end{align}

 With the goal of reaching a contradiction, let us suppose that $\mathcal A\setminus\mathcal A'$ has size continuum. Since $|\mathcal B| \leq \aleph_0 < \cf(\mathfrak c)$ and $$\mathcal A\setminus\mathcal A'=\bigcup_{B\in\mathcal B}\{p\in\mathcal A\setminus\mathcal A':B_p=B\},$$ we can invoke Lemma~\ref{lemma:2} to get some $B\in\mathcal B$ such that $\mathcal A_1:=\{p\in\mathcal A\setminus\mathcal A':B_p=B\}$ has $\mathfrak c$ elements. An analogous argument produces $B^*\in\mathcal B^*$ with the property that $\mathcal A_2:=\{p\in\mathcal A_1:B_p^*=B^*\}$ has size continuum.
 
 In summary, (\ref{eq:m}) gives us the following statement.
 
\begin{align}\label{eq:dagger}
\begin{split}
 \text{If } p\in\mathcal A_2,\text{ then }\vec p\restriction n\in B^*,\  \vec p(n)\in B,\text{ and }\\ g[S\cap\mathcal W(\vec p\restriction n,B^*)]\subseteq X\setminus B.
\end{split}
\end{align}

 By (\ref{eq:star}) and the fact that $\mathcal A_2$ is a pairwise disjoint family, we have that $\vec p[n]$ and $\vec q[n]$ are disjoint sets of size $n$. Moreover, $$\{\vec p[n]:p\in\mathcal A_2\}\subseteq\mathbb P_Y(K_0)$$ and, as a consequence, our inductive hypothesis gives us that there must be two different elements of $\mathcal A_2$, let's say $p$ and $q$, such that $\vec p[n]\cup\vec q[n]$ is 0-homogeneous. In particular, $p$ and $q$ are disjoint sets.
 
 According to (\ref{eq:dagger}), $\vec q\restriction n\in B^*$. On the other hand, the way we selected $p$ and $q$ implies that for all $i<n$, $\{\vec p(i),\vec q(i)\}\in K_0$. Therefore, $\vec q\restriction n\in\mathcal W(\vec p\restriction n,B^*)$. But we also have that $\vec q\restriction n\in S$ and so, again by (\ref{eq:dagger}), $\vec q(n)=g(\vec q\restriction n)\notin B$, which is a contradiction to the fact that $q\in\mathcal A_2$ (once again, see (\ref{eq:dagger})). This concludes the proof of the claim.
\end{claimproof}

 By basic cardinal arithmetic, a direct corollary of Claim~\ref{claim:2} is that $\mathcal A'$ has size continuum.
 
 The reals are a second countable topological space, so any subspace of a finite product of subspaces of $\mathbb R$ also has this property. In particular, $g\subseteq (X^n)\times X$ is a separable space, so we can assume there is an $h\subseteq g$ that is countable and dense in $g$ (hence, $h$ is a function of countable domain). Recalling the enumeration at (\ref{eq:enum_f}), $h=f_\eta$ for some $\eta<\mathfrak c$.
 
 Notice that $\dom(f_\eta)\subseteq\dom(g)=S=\{\vec p\restriction n:p\in\mathcal A\}$ and, moreover, that the set $$M:=\bigcup\left\{\img(e(p)):p\in\mathcal A\land\vec p\restriction n\in\dom(f_\eta)\}\cup\{\eta\right\}$$ is a countable subset of $\mathfrak c$. Since the family $\mathcal A'$ is pairwise disjoint (because it is a subset of $\mathcal A$), the function from $\mathcal A'$ into $\mathfrak c$ given by $p\mapsto e(p)(0)$ is one-to-one. Considering that $M$ is countable, there must exist some $t\in\mathcal A'$ such that $e(t)(0)$ is an upper bound of $M$, i.e.,

\begin{align}\label{eq:trebol}
 M\subseteq e(t)(0).
\end{align} 
 
\begin{claim}\label{claim:3}
 $\vec t(n)\in o(f_\eta,\vec t\restriction n)$.
\end{claim}

\begin{claimproof}
 Keeping in mind the definition of oscillation, we take arbitrary sets $U\in\mathcal O(\vec t\restriction n)$ and $V\in\tau_X(\vec t(n))$. By the fact $t\in\mathcal A'$, we have that $\vec t(n)\in o(g,\vec t(n))$, that is, $\vec t(n)\in\cl_X g[S\cap\mathcal W(\vec t\restriction n,U)]$. As a consequence of basic properties of the closure, there must be some $y\in S\cap\mathcal W(\vec t\restriction n,U)$ such that $g(y)\in V$.
 
 By Claim~\ref{claim:Wisopen}, $U^*:=U\cap\mathcal W(\vec t\restriction n,X^n)$ is an open subset of $X^n$. Moreover, for every $i,j<n$, $\{t(i),y(j)\}\in K_0$ and therefore $y\in U^*$. We then have that $(y,g(y))\in(U^*\times V)\cap g$, so the latter is a nonempty open set in $g$. By the density of $f_\eta$, we can find some $z\in\dom(f_\eta)\cap U^*$ such that $f_\eta(z)\in V$.
 
 Given that $z\in U^*$, one deduces that $z\in\mathcal W(\vec t\restriction n,U)$. As a consequence, $z$ is an element of $\dom(f_\eta)\cap\mathcal W(\vec t\restriction n,U)$ satisfying that $f_\eta(z)\in V$. This proves that $V$ has nonempty intersection with $f_\eta[\dom(f_\eta)\cap\mathcal W(\vec t\restriction n,U)]$, and so we are done with this claim.
\end{claimproof}

 In conclusion, $y_{e(t)(n)}=\vec t(n)\in o(f_\eta,\vec t\restriction n)$ and by (\ref{eq:trebol}) and the choice of $e$, $\eta\in M\subseteq e(t)(0)<e(t)(n)$. Hence, $\eta<e(t)(n)$.
 
 If we now apply (iii) of Claim~\ref{claim:1} to $\alpha:=e(t)(0)$, we get that $y_\alpha\in o(f_\eta,\vec t\restriction n)$ but $y_\alpha\notin o(f_\eta,z)$, so the oscillation $o(f_\eta,\vec t\restriction n)$ cannot be 1-homogeneous. In other words, two elements $x_0,x_1\in o(f_\eta,\vec t\restriction n)$ are such that $\{x_0,x_1\}\in K_0$ which, since $K_0$ is open in $X$, means we can find open sets $V_0$ and $V_1$ such that $(x_0,x_1)\in V_0\times V_1\subseteq K_0^\sharp$.
 
 Now since $x_0\in o(f_\eta,\vec t\restriction n)$, $V_0\in\tau_X(x_0)$, and obviously $X^n\in\mathcal O(\vec t\restriction n)$, $V_0$ must have nonempty intersection with $f_\eta[\dom(f_\eta)\cap\mathcal W(\vec t\restriction n,X^n)]$, in other words, there exists $p\in\mathcal A$ such that $\vec p\restriction n\in\dom(f_\eta)\cap\mathcal W(\vec t\restriction n,X^n)$ and $f_\eta(\vec p\restriction n)\in V_0$. But $f_\eta\subseteq g$, so $f_\eta(\vec p\restriction n)=g(\vec p\restriction n)=\vec p(n)$. Therefore, $\vec p(n)\in V_0$.

 By Claim~\ref{claim:Wisopen} (clearly, $\vec p\restriction n\in X^n$), $U:=\mathcal W(\vec p\restriction n,X^n)$ is open. Moreover, $\vec p\restriction n\in\mathcal W(\vec t\restriction n,X^n)$ implies that $\vec t\restriction n\in\mathcal W(\vec p\restriction n,X^n)$; in particular, $U\in\mathcal O(\vec t\restriction n)$. As a consequence, $x_1\in o(f_\eta,\vec t\restriction n)\subseteq\cl_X f_\eta[\dom(f_\eta)\cap\mathcal W(\vec t\restriction n,U)]$ and $V_1\in\tau_X(x_1)$; therefore, there is some $q\in\mathcal A$ such that $\vec q\restriction n\in\dom(f_\eta)\cap\mathcal W(\vec t\restriction n,U)$ and $\vec q(n)=g(\vec q\restriction n)\in V_1$.
 
 Since $V_0\times V_1\subseteq K_0^\sharp$, $\vec p(n)\in V_0$ and $\vec q(n)\in V_1$, we have that $V_0\cap V_1=\emptyset$ and also $p\neq q$. On the other hand, $\vec q\restriction n\in\mathcal W(\vec t\restriction n,U)$ implies that $\vec q\restriction n\in U$, and so for all $i,j<n$, $\{\vec p(i),\vec q(j)\}\in K_0$. Finally, $p\cup q$ forms a 0-homogeneous set and hence the proof of Todorc\v{e}vi\'{c}'s Lemma is complete.
\end{proof}

\begin{corollary}\label{coro:todorcevic}
 Let $X\subseteq\mathbb R$ and $K_0\subseteq[X]^2$ be open in $X$. Assume $X$ is not the union of less that continuum many 1-homogeneous sets. Then, under CH, we can produce a set $Y\subseteq X$ of size continuum such that the forcing notion $\mathbb P_Y(K_0)$ is ccc.
\end{corollary}

\begin{proof}
 Notice that we have the same hypotheses as in Lemma~\ref{lemma:todorcevic}. Hence, let $Y\subseteq X$ be such a set as in the conclusion of said lemma. Let $\mathcal A\subseteq\mathbb P_Y(K_0)$ be uncountable. We shall show that $\mathcal A$ cannot be an antichain, and thus prove the corollary.
 
 By Sanin's $\Delta$-system lemma (\cite[Theorem~1.5, p.~49]{kunen}), some $\mathcal A_0\subseteq\mathcal A$ of uncountable size has some root $r$. Immediately from this, the collection $\{p\setminus r:p\in\mathcal A_0\}$ is a pairwise disjoint subfamily of $\mathbb P_Y(K_0)$ and, by CH, we can assume it has $\mathfrak c$ elements. By Lemma~\ref{lemma:todorcevic}, there must exist two elements $p,q\in\mathcal A$ such that $p\setminus r\neq q\setminus r$ and $(p\setminus r)\cup(q\setminus r)$ is $0$-homogeneous. But then $p\cup q$ is also $0$-homogeneous, which by the definition of our forcing notion means that $p$ and $q$ are two different compatible elements of $\mathcal A$. Therefore, $\mathbb P_Y(K_0)$ must be ccc.
\end{proof}

\section{Some properties of $\sigma$-closed forcing notions}

 A couple of sections ago, we proved that all $\sigma$-closed forcing notions are proper. In this section we prove some more pertinent properties that we use in the final section.
 
 For the remainder of this section, we assume that $\mathbb P$ is a $\sigma$-closed forcing notion in the ground model $V$ and that $G$ is a $(V,\mathbb P)$-generic filter.
 
 It's a well-known fact that the $\sigma$-closedness of $\mathbb P$ implies that forcing with $\mathbb P$ does not add new sequences of length $\omega$ consisting of sets from the ground model; in other words, whenever $f\in V[G]$ satisfies $f\colon\omega\to V$, we obtain $f\in V$. Hence, intuition dictates that objects from the ground model that can be described using only countably many parameters remain the same in any generic extension given by $\mathbb P$. Parts (2) and (3) of our next result are a formalization of this idea for the specific cases of $\mathbb R$ and $\tau_{\mathbb R}$.
 
\begin{lemma}\label{lemma:A}

 The following three properties are true.

\begin{enumerate}
 \item $\mathbb P$ does not collapse $\aleph_1$, in other words, $\aleph_1^V=\aleph_1^{V[G]}$.
 \item $\mathbb P$ does not add new reals: $\mathbb R^V=\mathbb R^{V[G]}$. In particular, the continuum of $V$ is the same as $V[G]$.
 \item $\mathbb P$ does not create new open subsets of reals, that is, for all $X\in V$ satisfying $X\subseteq\mathbb R$, $(\tau_X)^V=(\tau_X)^{V[G]}$.
\end{enumerate}
\end{lemma}

\begin{proof}
 The fact that $\mathbb P$ does not collapse $\aleph_1$ is argued in detail in \cite[Lema~3.35, p.~58]{adrian}.
 
 By \cite[Lema~3.33, p.~58]{adrian}, $\mathcal P(\omega)^V=\mathcal P(\omega)^{V[G]}$. Thus, by using the fact that all real numbers are Dedekind cuts of the rational numbers, $\mathbb R^V=\mathbb R^{V[G]}$ and obviously $\mathfrak c^V=\mathfrak c^{V[G]}$.
 
 If, in $V$, $\mathcal B$ is the collection of open intervals in $\mathbb R$ with rational end-points, then $V[G]$ also models this fact. But then any open set of reals in $V[G]$ is the union of some sets in $\mathcal B$. By the previous paragraph, $V$ and $V[G]$ know the same subfamilies of $\mathcal B$, and thus any open set in $V[G]$ was already in $V$. In conclusion, $(\tau_\mathbb R)^V=(\tau_\mathbb R)^{V[G]}$, and the same follows for any subspace $X$ of the reals.
\end{proof}

 Now let us fix, in $V$, $X\subseteq\mathbb R$ and $K_0\subseteq[X]^2$ such that $K_0$ is open in $X$.
 
\begin{lemma}\label{lemma:B}
 For any set $A$, the following are equivalent.
\begin{enumerate}
 \item $A\in V$ and $V$ models that $A$ is a closed $1$-homogeneous subset of $X$.
 \item $A\in V[G]$ and $V[G]$ models that $A$ is a closed $1$-homogeneous subset of $X$.
\end{enumerate}
\end{lemma}

\begin{proof}
 To see that (1) implies (2), assume (1) and notice that there must be some $F\in V$ such that $V\models$ ``$F$ is closed in $\mathbb R$ and $A=F\cap X$.'' But then, by the previous lemma, $\mathbb R\setminus F\in(\tau_\mathbb R)^V=(\tau_\mathbb R)^{V[G]}$ and so, in $V[G]$, $F$ is a closed subset of $\mathbb R$ with $A = F\cap\mathbb R$. Moreover, if $V$ models that $[A]^2\cap K_0=\emptyset$ then, since $K_0\in V$, $V[G]$ also does this.
 
 Now assume (2). Similarly, some $F\in V[G]$ is such that, in $V[G]$, $F$ is closed in $\mathbb R$ and $A=F\cap X$. But by the previous lemma, $V$ and $V[G]$ have the same closed sets of reals. Hence, the fact $X\in V$ implies that $A\in V$ and $V$ also thinks that it is closed in $X$. By the same line of arguments as before, the property of $1$-homogeneity is absolute.
\end{proof}

\section{The relative consistency of OCA}

 For the climax of this work, we prove the final theorem.
 
\begin{theorem}
 PFA implies OCA.
\end{theorem}

\begin{proof}
 As usual, take $X\subseteq\mathbb R$ and $K_0\subseteq[X]^2$ open in $X$. Assume that $X$ is not a countable union of $1$-homogeneous sets, since in that case we are done. Hence, our goal is to use PFA to produce $Y$, an uncountable subset of $X$ which is $0$-homogeneous.
 
 We now consider the standard forcing notion that makes CH true in the generic extension, that is, $P$ is the set of all functions of some countable subset of $\omega_1$ into $\mathfrak c$, ordered by reverse inclusion. By the regularity of $\omega_1$, it is straightforward that $P$ is a $\sigma$-closed forcing notion, and thus everything we have done in the previous section applies here.
 
 If $G$ is a $(V,P)$-generic filter, then $\bigcup G$ is a surjective function from $\omega_1^{V[G]}$ onto $\mathfrak c^{V[G]}$, in other words, CH holds in $V[G]$.
 
 It follows from Lemma~\ref{lemma:A}(3) that $\mathbb P$ does not add new open sets to $X\times X$, i.e., $(\tau_{X\times X})^V = (\tau_{X\times X})^{V[G]}$. This equality evidently implies that $K_0$ remains open in $X$ in the generic extension.
 
\begin{nclaim}
 In $V[G]$, $X$ cannot be the union of less that continuum many $1$-homogeneous sets.
\end{nclaim}
 
\begin{claimproof}
 We do this by contradiction: since CH is true in $V[G]$, let us assume that $X$ is the countable union of $1$-homogeneous sets, that is, there exists $\{J_n:n<\omega\}\in V[G]$ such that every $J_n$ is $1$-homogeneous and $$X=\bigcup_{n<\omega}J_n.$$
 
 Let us fix an integer $n$. Working in $V[G]$, let $F_n$ be the topological closure in $X$ of the set $J_n$. Clearly, $X = \bigcup\{F_k : k<\omega\}$. On the other hand, by Lemma~\ref{lemma:3}(2), $F_n$ is $1$-homogeneous. Hence, according to Lemma~\ref{lemma:B}, $F_n$ is a member of the ground model. Now, the fact that $P$ is $\sigma$-closed gives $\{F_k : k<\omega\}\in V$ and so, in $V$, $X$ can be written as a countable union of $1$-homogeneous sets, an obvious contradiction.
\end{claimproof}

 In summary, $V[G]$ believes that $X$ and $K_0$ comply with the hypothesis of Lemma~\ref{lemma:todorcevic}. Hence we are cleared to apply Corollary~\ref{coro:todorcevic} to produce a set $Y\subseteq X$ of size $\aleph_1$ such that $\mathbb P_Y(K_0)$ is a ccc forcing notion. We can rephrase this statement by saying that there is a one-to-one function $f\colon\omega_1\to X$ in such a way that $\mathbb P_{\img(f)}(K_0)$ is ccc.

 Since $G$ was an arbitrary $(V,P)$-generic filter, we deduce that there are $P$-names, $\dot f$ and $\dot Q$, in such a way that $1\Vdash$ ``$\dot f\text{ is a one-to-one function from }\omega_1\text{ into }X \text{ and }\dot Q =\mathbb P_{\img(\dot f)}(K_0)$ is ccc''. Then Lemma~\ref{lemma:sigma} gives us that $1\Vdash$ ``$\dot Q$ is a proper forcing notion''.
 
 By Lemma~\ref{lemma:shelah}, the iteration $\mathbb P:=P*\dot Q$ is proper. For every $\alpha<\omega_1$, routine arguments show that $$D_\alpha:=\{(p,\dot q)\in\mathbb P:\exists x\in X\ \exists\xi\in\omega_1\setminus\alpha\ (p\Vdash x=\dot f(\xi)\in\dot q)\}$$ is a dense subset of $\mathbb P$.
 
 By PFA, there exist $F$, $\{(p_\alpha,\dot q_\alpha):\alpha<\omega_1\}\subseteq\mathbb P$, $\{y_\alpha:\alpha<\omega_1\}\subseteq X$, and $e\colon\omega_1\to\omega_1$ in such a way that $F$ is a filter in $\mathbb P$ and, for each $\alpha<\omega_1$, one has $e(\alpha)\geq\alpha$ and $p_\alpha\Vdash y_\alpha=\dot f(e(\alpha))\in\dot q_\alpha$. Then, $Y := \{y_\alpha : \alpha<\omega_1\}$ is as required.
\end{proof}

\bibliographystyle{plain}


\end{document}